\documentclass[11pt,fleqn]{amsart}
\usepackage{amsmath}
\usepackage{amssymb}
\usepackage{amsthm}
\theoremstyle{plain}
\newtheorem{theo}{Theorem}
\newtheorem*{theo*}{Theorem}
\newtheorem{prop}{Proposition}

\newtheorem*{lem*}{Lemma}

\theoremstyle{definition}
\newtheorem{lem}{Lemma}[section]

\newtheorem{prob}{Problem}[section]

\theoremstyle{remark}

\newtheorem*{ack}{Acknowledgements}
\newtheorem*{rem*}{Remark}
\newtheorem{rem}{Remark}[section]
\begin{document}

\title{On invariant sets in Lagrangian graphs}
\author{Xiaojun Cui \hspace{0.6cm} Lei Zhao}
\address{Xiaojun Cui \endgraf
Department of Mathematics,\endgraf Nanjing University,\endgraf
Nanjing, 210093,\endgraf Jiangsu Province, \endgraf People's
Republic of China.} \email{xjohncui@gmail.com}
\curraddr{Mathematisches Institut, Albert-Ludwigs-Universit\"{a}t, Eckerstrasse 1, 79104, Freiburg im Breisigau, Germany}

\address{Lei Zhao \endgraf
 IMCCE,\endgraf 77 Avenue Denfert-Rochereau \endgraf   Paris, 75014 \endgraf \& Department of Mathematics, \endgraf Paris 7 University \endgraf Paris 75013}
 \email{zhaolei@imcce.fr}
\thanks{The first author is supported by National Natural Science Foundation of China
 (Grant 10801071) and  research fellowship for postdoctoral researchers
 from the Alexander von Humboldt Foundation.}
\maketitle

\abstract In this exposition, we show that a Hamiltonian is always constant on a compact invariant connected subset which lies in  a Lagrangian graph provided that the Hamiltonian and the graph are smooth enough. We also provide some counterexamples for the case that the Hamiltonians are not smooth enough.

\endabstract

\section{Introduction}

Let $M$ be a closed, connected $C^{\infty}$  manifold of dimension $d$, and $T^*M$ be the cotangent bundle of $M$. We always assume that Hamiltonian $T^*M \rightarrow \mathbb{R}$ is $C^r$ smooth ($r \geq 1$). We denote the associated Hamiltonian vector field and Hamiltonian flow by $X_H$ and $\phi^t_H$ respectively.

In Hamiltonian dynamics, the following result is well known:

\textbf{Let $\Gamma$ be an invariant (under the Hamiltonian flow $\phi^t_H$) $C^1$ Lagrangian graph, then H is constant on $\Gamma$.}

In fact, if $\Gamma$ is only Lipschitz, the result still holds \cite{So}, i.e.,
\begin{prop}
Let $\Gamma$ be an invariant (under the Hamiltonian flow $\phi^t_H$) Lipschitz Lagrangian graph, then H is constant on $\Gamma$.
\end{prop}

We always assume  the Lagrangian graphs we consider are at least $C^1$, unless other stated. After this proposition, it is naturally then to pose the following problem:
\begin{prob}
If $\Lambda$ is a compact, connected, invariant (under $\phi^t_H$) set, and $\Lambda \subseteq \Gamma$, then is $H$ constant on $\Lambda$?
\end{prob}
In the case $\Lambda \neq \Gamma$, the answer to this problem is not obvious, since the structure of $\Lambda$ could be very complicated. We will study this problem concretely in this short exposition.

 We denote the projection of $\Lambda$ into $M$ by $\Lambda_0$.

More precisely, we have:
\begin{theo} If  $h(q):=H(q,\Gamma(q)) \in C^{d^{\prime},s}(M,\mathbb{R})$  with $d^{\prime} \geq d$, or $d^{\prime}=d-1$ and $s=1$, then H is constant on $\Lambda$.
\end{theo}

\begin{rem}
    Actually the conclusion of the former theorem still holds under weaker conditions, for example $h \in C^{d-1, Zygmund}$, i.e., the $d-1$ order derivatives of $h$ is smooth in the sense of Zygmund (see \cite{zy} for details).
\end{rem}

We say $\Gamma$  is a Lipschitz Lagrangian graph, if $\Gamma$ coincides with the differential of a $C^{1,1}$ function locally. Then, we have

\begin{rem}
    In the case of 1 degree of freedom,  one can show that if $\Lambda$ is a compact, connected, invariant set under $\phi^t_H$, and $\Lambda$ lies in a Lipschitz Lagrangian graph, then $H|_{\Lambda}$ is constant.
    \end{rem}

    \begin{rem}
If $\Lambda_0$ admits some special structures, e.g., Lipschitz lamination, lower Hausdorff dimension, semi-analytic or semi-algebraic, then $H$ is still constant on $\Lambda$ under some weaker (than Theorem 1) smooth hypothesis. We refer to \cite{fa},\cite{cr}, for more details.
\end{rem}

Among these cases stated in Remark 1.3, the most interesting case is

\begin{theo}
If for any two points in $\Lambda$, there is a rectifiable  path in $\Lambda$ which connects them, then $H$ is constant on $\Lambda$.
\end{theo}

In the case that $H$ is not so smooth, we have the following:

\begin{theo} Assume that $d \geq 2$. For $d^{\prime}< d-1$ and $s \in [0,1]$ or $d^{\prime}=d-1$ and $s \in [0,1)$ , there exist examples with $H\in C^{d^{\prime},s}(T^*M,\mathbb{R})$ such that $H$ is non constant on $\Lambda$.
\end{theo}
\begin{rem}
These examples show that the condition in Theorem 1 is optimal in some sense.
\end{rem}

\section{Proof of Theorem 1}
 The following lemma is an easy consequence of the flow-invariance of $\Lambda$:

    \begin{lem}
   If  $\Gamma$ is  $C^{1}$ smooth, then $\Lambda_0$ is contained in the critical set of  $h$.
\end{lem}
\begin{proof}
We will prove $dh(q_0)=0$ for any point $q_0 \in \Lambda_0$. For this, we only need to show that $dh(q_0) \cdot v=0$ for any $ v\in T_{q_0}M$. Now we also regard $\Gamma$ as a map from $M$ to $T^*M$, then $dh(q_0) \cdot v=dH(q_0,\Gamma(q_0)) \cdot \Gamma_* v$, here $\Gamma_* v \in T_{(q_0,\Gamma(q_0))}{\Gamma}$. Since $\Lambda$ is invariant under the flow $\phi^t_H$, we have $X_{H}(q_0,\Gamma(q_0)) \in T_{(q_0,\Gamma(q_0))}{\Gamma}$. Note that $T_{(q_0,\Gamma(q_0))}{\Gamma}$ is a Lagrangian subspace, we have
$$dh(q_0) \cdot v=dH(q_0,\Gamma(q_0)) \cdot \Gamma_* v=-\omega(X_H,\Gamma_* v)=0.$$

\end{proof}

Clearly, we may generalize Lemma 2.1 to
    \begin{lem}  If $\Gamma$ is Lipschitz,  then every differentiable points contained in $\Lambda_0$ is critical for $h$.
    \end{lem}

      Now we begin to prove Theorem 1.

 Suppose $H$ is not constant on $\Lambda$. This means that $h$ is not constant on its  critical point set $\Lambda_0$. Note that $\Lambda_0$ is connected, so the Lebesgue measure of the set of critical values of $h$ is positive. This contradicts to Bates' improved Morse-Sard's theorem \cite{ba}.

    \section{Proof of Theorem 2}

    Of course, it is a direct consequence of Norton's improved Morse-Sard's theorem \cite{cr}. However, we present  a slightly different proof here.

      For any two points $(q_1,p_1),(q_2,p_2)$ on $\Lambda$, denote by $\beta$ the rectifiable path connects them.  Note that $\beta \in \Lambda$, and $\Lambda$ is invariant, so $dH \cdot \dot{\beta}(t)=0$, at each differential point, (here, we choose $t$ as the parameter of arc length). Thus $$H((q_2,p_2))-H((q_1,p_1))=\int dH \cdot \dot{\beta}(t)=0.$$

\section{Proof of Theorem 3}

    In \cite{w}, Whitney constructed a function $f(q) \in C^{d-1}$ on $d$ ($\geq 2$) dimension manifold $M$ such that there exists a connected set  $\Lambda_0$ with $df(q)=0$ for every $q \in \Lambda_0$, but $f$ is not constant on $\Lambda_0$. In \cite{Alec}, Norton showed more in this direction the existence of a large class of Whitney-type examples for $f \in C^{d-1,s}$ with $0 \le s <1 $.

    By using these Whitney-Norton type examples, we can construct examples Theorem 3 required.

In fact, for any $s\in [0,1)$, there exists a $C^{d-1,s}$ function $f(q)$ and a connected subset $\Lambda_0 \subset M $ such that $df(q)=0$,  $\forall q \in \Lambda_0$, but $f(q)$ is not constant on $\Lambda_0$. Moreover, we may assume that $\Lambda_0$ is contained in a coordinate neighborhood $U$, by changing $f$ outside if necessary. Shrinking $U$ if necessary, we may introduce an auxiliary $C^{\infty}$ Riemannian metric $g$ such that $g$ is Euclidiean on $U$.

    Now we define the Hamiltonian:
    $$H(q,p)=f(q)+\frac{1}{2}|p|^2,$$
    where
    $$q=(q_1,q_2,\cdots,q_d), p=(p_1,p_2,\cdots,p_d)$$
     are local coordinates of $T^*M$, and $|\cdot |$ is induced by the Riemannian metric  $g$. The Hamiltonian equation is:

$$
\dot{q}=\dfrac{\partial H(q,p)}{\partial p}=p,  \,\,\,\,\,\,
 \dot{p}=-\dfrac{\partial H(q,p)}{\partial q}=h(q).
$$

    Let $\Lambda=(\Lambda_0,0)$, then $\Lambda$ is contained in the zero section of $T^*M$. It is easy to check that $\Lambda$ is invariant under the flow $\phi^t_H$. But $H|_{\Lambda}=h|_{\Lambda_0}$ is not constant by the definition of $f$.

    \begin{rem}
    If, we take Hamiltonian to be
    $$H(q,p)=f(q)+\frac{1}{2}|p-\Gamma|^2,$$
    here $\Gamma$ is any Lagrangian graph, then the required invariant critical set $\Lambda \subset \Gamma$.
    \end{rem}

    \begin{rem}
    In this example, the invariant set $\Lambda$ consists only of fixed points. In fact, we can also construct examples such that $\Lambda$ support non-Dirac measures:

    For instance, consider the standard $4$-torus. Let $f(q_1,q_2,q_3)$ be a function of Whitney-Norton type on $3$-sub-torus, (denote the associated connected critical set by $\Lambda_1$), as discussed above. Now let the Hamiltonian be $$H(q_1,q_2,q_3,q_4,p_1,p_2,p_3,p_4)=f(q_1,q_1,q_3)+\dfrac{1}{2}(p_1^2+p_2^2+p_3^2+(p_4+1)^2),$$
    \end{rem}
    then $\Lambda_0=\Lambda_1 \times \mathbb{T}$ is the required projected invariant set, and
    $$\Lambda=\{(q_1,q_2,q_3,q_4,0,0,0,0): (q_1,q_2,q_3) \in \Lambda_1\}.$$
     Clearly, $\Lambda$ is contained in the zero section, and the Hamiltonian flow is not stationary on $\Lambda$.

\section{Problems}
In the example in Theorem 3, the section is $C^{\infty} $, but the Hamiltonian $H$ is finite smooth. It is more interesting if one can construct counterexamples with infinitely smooth Hamiltonian and finite smooth Lagrangian graph. For this purpose, we pose the following problems:

\begin{prob}
Can one construct an explicit example of $H$ of $C^{\infty}$, which admits a compact, connected invariant set $\Lambda$ in a Lagrangian graph $\Gamma$ of finite smooth, such that $H$ is not constant on $\Lambda$?
\end{prob}

We call a graph $\Gamma$ is $C^{0,s}$ Lagrangian, if $\Gamma$ coincides with a  differential of a $C^{1,s}$ function locally.  As a negative side of Proposition 1, we also pose

\begin{prob}
Can one construct an explicit example of $H$, which admits an invariant $C^{0,s}$ (here $0 \leq s <1$) Lagrangian graph $\Gamma$, such that $H$ is not constant on $\Gamma$?
\end{prob}

\begin{rem}
For Tonelli Hamiltonians, solutions of the associated Hamilton-Jacobi equation have the following nice property: a $C^1$ solution must be $C^{1,1}$, \cite{fa1}. So, if one can construct a $C^{0,s}$ $(0 \leq s <1))$, non-Lipschitz  invariant (under the flow of $\phi^t_H$, H is Tonelli Hamiltonian) Lagrangian graph $\Gamma$, then $H$ is not constant automatically.
\end{rem}
\section{Appendix}
In this appendix, we give a proof of Proposition 1, which is slightly different from \cite{So}.

Let $h$ be the function as in Theorem 1, then $h$ is a Lipschitz function on $M$, and $dh=0$ at any differentiable point. For any two points $q_0,q_1$, we can choose an absolutely continuous curve $\gamma:[0,1] \rightarrow M$ with $\gamma(0)=q_0,\gamma(1)=q_1$ and $h$ is differentiable on $\gamma$ almost everywhere. Hence, $h(q_0)=h(q_1)$. Thus,  $h$ constant on $M$, and $H$ is constant on $\Gamma$, consequently.a

\begin{ack}
The authors thank Patrick Bernard, Alain Chenciner for helpful discussions. It is Bernard suggested that the problem we consider is related to Sard's theorem.

This paper was completed when the  first author
visited Albert-Ludwigs-Universit\"{a}t as a postdoctoral researcher, supported
by a fellowship from the Alexander Von Humboldt Foundation. The first
 author would like to thank Professor V. Bangert and Mathematisches Institut at Albert-Ludwigs-Universit\"{a}t for the hospitality.
\end{ack}


\begin{thebibliography}{99}
   \bibitem{ba} S. Bates, \textsl{Toward a precise smoothness hypothesis in Sard¡¯s theorem}, Proc. Amer. Math. Soc., 117 (1), (1993), 279-283.

       \bibitem{fa1} A. Fathi, \textsl{Regularity of C$^1$ solutions of the Hamilton-Jacobi equation},
Annales de la facult¨¦ des sciences de Toulouse, S¨¦r. 6, 12 (4), (2003), 479-516.


    \bibitem{fa} A. Fathi, A Figalli $\&$ L. Rifford, \textsl{On the Hausdorff dimension of the Mather quotient}, Preprint, (2007), 56pp..

   \bibitem{cr} A. Norton, \textsl{A critical set with nonnull image has large Hausodrff dimension}, Trans. Amer. Math. Soci., 296 (1), (1986), 367-376.

   \bibitem{Alec} A. Norton, \textsl{Functions not constant on fractal quasi-arcs of critical points}, Proc. Amer. Math. Soc, 106 (2), (1989),  397-405 .

   \bibitem{zy} A. Norton, \textsl{The Zygmund Morse-Sard Theorem}, Journal of Geometric Analysis, 4 (3), (1994), 403-424.

   \bibitem{So} A. Sorrentino, \textsl{On the integrability of Tonelli Hamiltonians}, Preprint, (2009), 19pp..

  \bibitem{ae} H. Whitney, \textsl{Analitic extensions of differentiable functions defined in closed sets}, Trans. Amer. Math. Soci., 36 (1934), 63-89.

   \bibitem{w} H. Whitney, \textsl{A function not constant on a connected set of critical points}, Duke Math., 1 (4) (1935), 514-517.


   \end{thebibliography}
\end{document}